\documentclass[a4paper,oneside,english]{amsart}
\usepackage[T1]{fontenc}
\usepackage[utf8]{inputenc}
\usepackage{babel}
\usepackage{amstext}
\usepackage{amsthm}
\usepackage{amsmath}
\usepackage{xcolor}

\usepackage{amssymb}
\usepackage{setspace}
\usepackage{mathrsfs}
\usepackage[hidelinks,colorlinks=true,linkcolor=blue,citecolor=blue]{hyperref}
\usepackage[colorinlistoftodos,prependcaption,textsize=tiny]{todonotes}
\usepackage{lipsum}
\usepackage{xargs}
\usepackage{graphicx}
\renewcommand\eqref[1]{(\ref{#1})}
\DeclareUnicodeCharacter{FB00}{ff}
\graphicspath{{figures/}}

\numberwithin{equation}{section}
\theoremstyle{plain}
\newtheorem{theorem}{Theorem}[section]
\newtheorem{proposition}[theorem]{Proposition}
\newtheorem{corollary}[theorem]{Corollary}
\newtheorem{lemma}[theorem]{Lemma}

\theoremstyle{definition}
\newtheorem{definition}[theorem]{Definition}
\newtheorem{remark}[theorem]{Remark}

\newcommand{\Hh}{\mathbb{H}^1}
\newcommand{\CC}{\mathrm{CC}}

\newcommand{\R}{\mathbb{R}}

\title[CC--geodesic Kakeya sets]{A short note on the dimension of the CC--geodesic Kakeya sets in the first Heisenberg group}
\author[M. Chatzakou]{Marianna Chatzakou}
\address{
  Marianna Chatzakou:
  \endgraf
  Department of Mathematics: Analysis, Logic and Discrete Mathematics
  \endgraf
  Ghent University, Ghent, Belgium
  \endgraf
  {\it E-mail address} {\rm marianna.chatzakou@ugent.be}
}

\keywords{Kakeya sets, CC--geodesics, Heisenberg group, Hausdorff
dimension}

\dedicatory{In memory of Yannis Sarantopoulos}

\begin{document}

\begin{abstract}
We study CC--geodesic Kakeya sets in the first Heisenberg group,
namely Borel sets \(E\) such that, for every unit-speed 
CC--geodesic segment of length \(1\) issuing from the identity, some
left translate of the segment is contained in \(E\). The natural
analogue of the Kakeya conjecture would predict full Heisenberg
Hausdorff dimension \(4\) for such sets. We show that this prediction
fails: the sharp lower bound for their Heisenberg Hausdorff dimension
is \(3\), and it remains sharp even among compact CC--geodesic Kakeya
sets. By adjoining a Lebesgue-null set of full Heisenberg Hausdorff
dimension, we  obtain a CC--geodesic Kakeya set of full Heisenberg
Hausdorff dimension \(4\) and zero Lebesgue measure. Finally, when one prescribes only geodesic segments with one fixed
nonzero curvature parameter \(\kappa\), rather than segments of all
curvatures, the condition is weaker. For every
\(\kappa\in(0,2\pi]\), we construct a compact
curvature-\(\kappa\) Kakeya set of zero Lebesgue measure whose
Euclidean and Heisenberg Hausdorff dimensions are both equal to \(1\).
\end{abstract}

\maketitle

\section{Introduction}\label{sec:introduction}

The classical Kakeya problem asks how small a set can be if it contains a
unit line segment in every direction.  More precisely, a compact  set
\(E\subset\mathbb R^n\) is a Kakeya set if, for every
\(e\in S^{n-1}\), there exists \(a\in\mathbb R^n\) such that
\[
a+[0,1]e\subset E.
\]
Besicovitch showed that such sets may have Lebesgue measure zero already in
the plane \cite{Bes28}.  The Kakeya conjecture asserts nevertheless that every
Kakeya set in \(\mathbb R^n\) has Hausdorff dimension \(n\).  This was proved
in the plane by Davies \cite{Dav71}; after a long sequence of improvements in
three dimensions, including \cite{KLT00,KZ19}, Hong Wang and Joshua Zahl
recently proved that every Kakeya set in \(\mathbb R^3\) has both Hausdorff
and Minkowski dimension \(3\) \cite{WZ25}. The literature on the Kakeya problem and its connections with harmonic
analysis is vast, and we do not attempt to provide a comprehensive
bibliography here.  For broader surveys of the problem and its
development, we refer the reader to \cite{Wol99Survey,KT02Survey,Zah26Survey}.

The main motivation of the present note is to formulate and study the
corresponding question in the first Heisenberg group \(\Hh\), equipped with
its Carnot--Carath\'eodory geometry. Since the Euclidean definition requires a set to contain a translate of every unit segment issuing from the origin, in \(\Hh\) we accordingly require a set to contain a left translate of every unit-speed  CC--geodesic segment issuing from the identity; see Definition~\ref{def:cc-kakeya} below.  The existence of a Lebesgue-null set satisfying this condition was
posed as an open question by A.~Lukyanenko
\cite{LukyanenkoQuestions}.  More precisely, Definition~\ref{def:cc-kakeya} is the
sub-Riemannian analogue of the classical Euclidean definition of a Kakeya set, where we make the following two natural replacements:
\begin{itemize}
\item Euclidean unit segments are replaced by unit-speed 
CC--geodesic segments;
\item Euclidean translations are replaced by Heisenberg left translations,
which are isometries for the Carnot--Carath\'eodory distance.
\end{itemize}

Since the homogeneous dimension of \(\Hh\) is \(Q=4\), the direct analogue
of the Euclidean Kakeya conjecture would predict that every CC--geodesic
Kakeya set has full Heisenberg Hausdorff dimension \(4\).  A fact that at
first glance supports this conjecture is that, if one uses only horizontal
unit segments issuing from the identity, rather than all unit-length
CC--geodesic segments issuing from the identity, then the sharp lower bound
for the Heisenberg Hausdorff dimension is \(3\).  Liu proved this sharp lower
bound and showed that it is attained by the horizontal plane \cite{L22}.
F\"assler, Pinamonti, and Wald later obtained the same sharp lower bound from
a Heisenberg Kakeya maximal inequality \cite{FPW}. See also the recent bilinear Kakeya inequality of Galanos
\cite{Gal26}. These horizontal
segments are precisely the zero-curvature members of our family.  Hence
either result immediately gives the lower bound \(3\) for the stronger class
considered here.  Its sharpness, however, does not follow from the
horizontal-plane example, because the horizontal plane does not contain the
required nonzero-curvature CC--geodesic segments, and one might be tempted
to believe that the lower bound in this case is \(4\), that is, that the
CC--geodesic Kakeya conjecture in \(\Hh\) is true.

Our first result
shows that this prediction is, however, false; the optimal lower bound is instead
\(3\). Precisely, we prove the following:

\begin{theorem}\label{thm:main}
Every CC--geodesic Kakeya set in \(\Hh\) has Heisenberg Hausdorff
dimension at least \(3\), and this lower bound is sharp.
\end{theorem}
In particular, the additional nonzero-curvature geodesics do not
raise the sharp lower bound above the one already known for horizontal line
segments. For the proof of Theorem \ref{thm:main}  we construct a different set, formed by two smooth
Euclidean surfaces, which contains all the required geodesic segments and has
Heisenberg Hausdorff dimension \(3\).

A related but distinct Euclidean problem was studied by
F\"assler--Orponen \cite{FO23} and, independently, by Katz--Wu--Zahl
\cite{KWZ23}. They consider Kakeya sets associated with the \(SL_2\)
family of lines in \(\mathbb R^3\), and prove that such sets have
Euclidean Hausdorff dimension \(3\). One part of this family consists
of the lines
\[
(a,b,0)+\operatorname{span}(c,d,1),
\qquad ad-bc=1.
\]
Under the change of coordinates
\[
\Xi(x,y,t):=(x,y,t/2),
\]
such a line becomes a horizontal Heisenberg line:
\[
\Xi\bigl((a,b,0)+s(c,d,1)\bigr)
=
(a,b,0)\cdot(sc,sd,0).
\]
Thus the connection with horizontal Heisenberg lines is explicit,
although both the prescribed directions and the Hausdorff dimensions
in these results are Euclidean.

We next consider whether the sharp lower bound can be attained by a
compact CC--geodesic Kakeya set; although Liu's sharpness example
\(\Pi\) is unbounded, the closed disk \(D_2\subset\Pi\) is a compact
Kakeya set for the horizontal definition and has Heisenberg Hausdorff
dimension \(3\); see Remark~\ref{rem:D_2}. For the full
CC--geodesic definition, constructing a compact sharpness example is
more involved. Nevertheless, we prove the following.

\begin{theorem}
    \label{thm:main4}
  Every compact CC--geodesic Kakeya set in \(\Hh\) has Heisenberg Hausdorff
dimension at least \(3\), and this lower bound is sharp.  
\end{theorem}

On the other hand, although \(3\) is the smallest
possible Heisenberg Hausdorff dimension, there exists a CC--geodesic
Kakeya set of full Heisenberg Hausdorff dimension \(4\) but Lebesgue measure zero. This can be done simply by adjoining a Lebesgue-null set of Heisenberg Hausdorff
dimension $4$. In particular, we show the following: 

\begin{corollary}\label{thm:main2}

There exists a CC--geodesic  Kakeya set in \(\Hh\) with full
Heisenberg Hausdorff dimension \(4\) but Lebesgue measure $0$.
\end{corollary}

Second, we study what happens when we 
fix one nonzero curvature \(\kappa\), and we see that in this case  the problem changes substantially;
a single extended CC--geodesic trajectory contains, after suitable left
translations, all the required unit segments.  Following Katz--Wu--Zahl
\cite[\S~1.2.1]{KWZ23}, we use the term \emph{near miss at dimension
\(d\) for the Kakeya conjecture} for an object which shares certain
properties with Kakeya sets but has dimension \(d<3\).  Two standard
near misses at dimension \(5/2\) are the Heisenberg example of
Katz--{\L}aba--Tao \cite{KLT00} and the \(SL_2\) example of
Katz--Zahl \cite{KZ19}.  The former is a subset of
\(\mathbb C^3\) and should not be confused with the first real
Heisenberg group \(\Hh\simeq\mathbb R^3\) considered in the present
work. 
In our setting, for \(\kappa\in(0,2\pi]\), we define curvature-\(\kappa\) Kakeya sets by requiring all prescribed CC--geodesic segments to have curvature \(\kappa\), and we show that such a set can have Hausdorff dimension as small as \(1\). In particular, 
\begin{theorem}\label{thm:main3}
For every \(\kappa\in(0,2\pi]\), there exists a curvature-\(\kappa\) Kakeya set in \(\Hh\) which is compact and has Euclidean and Heisenberg Hausdorff dimension \(1\). In other words, for every \(\kappa\in(0,2\pi]\) there exists a near miss at dimension \(1\) for the CC--geodesic Kakeya conjecture in \(\Hh\).
 \end{theorem}

The paper is organised as follows.  Section~\ref{sec:preliminaries} recalls
the Heisenberg geometry, describes the minimising CC--geodesics, and compares
Definition~\ref{def:cc-kakeya} with earlier notions of Kakeya sets in the Heisenberg group.  In
Section~\ref{sec:proof} we prove Theorem \ref{thm:main} and Corollary \ref{thm:main2}.  The fixed-curvature
construction is given in Section~\ref{sec:further-results}. The proof of Theorem \ref{thm:main4} is contained in Section~\ref{sec:further-results} since it uses results from this section.

\section{Preliminaries}\label{sec:preliminaries}
We briefly recall some basics about  \(\Hh\). For a general
background, see \cite{FR16,Fol89,Ste93,Tha98}; for a more geometric treatment,
see  \cite{Ladonna}. We realise $\Hh$ as the manifold \(\R^3\) endowed with the
group law
\[
(x,y,t)\cdot(x',y',t')
:=
\Bigl(x+x',\,y+y',\,t+t'+\tfrac12(xy'-yx')\Bigr).
\]
The \textit{(canonical) left-invariant horizontal vector fields}  are
\[
X:=\partial_x - \tfrac{y}{2}\partial_t,
\qquad
Y:=\partial_y + \tfrac{x}{2}\partial_t\,.
\]
For every \(p\in\mathbb H^1\), define $\Delta_p:=\operatorname{span}\{X(p),Y(p)\}
    \subset T_p\mathbb H^1.$
The family \(
    \Delta:=\{\Delta_p\}_{p\in\mathbb H^1}
\)
 is  the
\emph{horizontal distribution}.
 A \(C^1\) curve
\(\gamma(s)=(x(s),y(s),t(s))\) is \emph{horizontal} if
\(\dot\gamma(s)\in\Delta_{\gamma(s)}\)  for every \(s\), equivalently, if 
there are functions \(a,b\)  such that
\[
\dot\gamma(s)
=a(s)X(\gamma(s))+b(s)Y(\gamma(s)).
\]
Comparing the coefficients of \(\partial_x\) and \(\partial_y\) gives
\(a(s)=\dot x(s)\) and \(b(s)=\dot y(s)\); comparing the coefficient  of
\(\partial_t\) then gives
\begin{equation}\label{eq:horiz}
\dot t(s)
=
\frac{1}{2}\bigl(x(s)\dot y(s)-y(s)\dot x(s)\bigr).
\end{equation}
Thus,  a \(C^1\) curve is horizontal if and only if it satisfies
\eqref{eq:horiz}. Its Carnot--Carath\'eodory length is
\[
L_{\CC}(\gamma)
:=
\int\sqrt{\dot x(s)^2+\dot y(s)^2}\,ds,
\] 
and the Carnot--Carath\'eodory distance is
\[
d_{\CC}(p,q)
:=
\inf\{L_{\CC}(\gamma):\gamma\text{ is horizontal from }p\text{ to }q\}.
\]

The study  of the Carnot--Carathéodory geodesics (in short  CC--geodesics) in the Heisenberg group goes
back to Gaveau \cite{Gav75,Gav77}. Later on, Monti 
\cite{Mon00} gives their explicit description. In the present normalisation, the
unit-speed minimising  CC--geodesics issuing from the identity $e \in \Hh$ admit the
following  description; see for instance,
\cite[\S~1.4.1 and Equation~(1.4.2)]{Ladonna}.
For \(k\in\mathbb R\), set
\[
I_k:=
\begin{cases}
\left[0,\dfrac{2\pi}{|k|}\right], & k\neq 0,\\[6pt]
[0,\infty), & k=0.
\end{cases}
\]
For \(\theta\in[0,2\pi)\) and \(k\in\mathbb R\), define \(\widetilde\gamma_{\theta,k}:I_k\to\Hh\) by
 \[
\widetilde\gamma_{\theta,k}(s)
=
\bigl(x_{\theta,k}(s),y_{\theta,k}(s),t_k(s)\bigr),
\]
where, for \(k\neq0\),
\begin{equation}\label{curves}
\begin{aligned}
x_{\theta,k}(s)
&=
\frac{\cos\theta(\cos(ks)-1)-\sin\theta\sin(ks)}{k},\\
y_{\theta,k}(s)
&=
\frac{\sin\theta(\cos(ks)-1)+\cos\theta\sin(ks)}{k},\\
t_k(s)
&=
\frac{ks-\sin(ks)}{2k^2},
\end{aligned}
\end{equation}
and, for  \(k=0\), \begin{equation}\label{eq:kappa=0}
\widetilde\gamma_{\theta,0}(s)
=
\bigl(-s\sin\theta,s\cos\theta,0\bigr).
\end{equation}
  Consequently, the unit-speed minimising \footnote{Here minimising means that \(\gamma_{\theta,k}\) has the least
CC length among all horizontal curves joining its endpoints.  Indeed,
the cut time is \(2\pi/|k|\) for \(k\neq0\) and \(+\infty\) for \(k=0\);
hence, since \(|k|\leq2\pi\), the interval \([0,1]\) does not extend
beyond the cut time, and \(\gamma_{\theta,k}\) is minimising on
\([0,1]\); see \cite[Definition~11 and \S~6.2]{BBN12}.} CC--geodesics of length \(1\)
issuing from \(e\) are precisely
\[
\gamma_{\theta,k}
:=
\left.\widetilde\gamma_{\theta,k}\right|_{[0,1]},
\qquad
\theta\in[0,2\pi),\quad k\in[-2\pi,2\pi].
\]
 A direct computation gives
 \[
\sqrt{\dot x_{\theta,k}(s)^2+\dot y_{\theta,k}(s)^2}=1,
\qquad s\in[0,1],
\]
so \(\gamma_{\theta,k}\) is parametrised by CC arclength and 
\[
L_{\CC}(\gamma_{\theta,k})
=
\int_0^1
\sqrt{\dot x_{\theta,k}(s)^2+\dot y_{\theta,k}(s)^2}\,ds
=1.
\]

\smallskip

Throughout, \(\mathcal L^n\) denotes the \(n\)-dimensional Lebesgue measure
on \(\mathbb R^n\). Under the identification \(\mathbb H^1\simeq\mathbb R^3\),
the measure \(\mathcal L^3\) is both a left- and right-invariant Haar measure
on \(\mathbb H^1\). We therefore take \(\mathcal L^3\) as our fixed Haar
measure. By \(\dim_{\mathcal H}^{\Hh}\) we denote Hausdorff dimension computed
with respect to either the Kor\'anyi metric \(d_K\) or the
Carnot--Carath\'eodory metric \(d_{\CC}\). These metrics are
bi-Lipschitz equivalent; see, for example,
\cite[Proposition~5.1.4]{BLU07}. Thus there exists \(L\geq1\) such that
\[
L^{-1}d_K(x,y)\leq d_{\CC}(x,y)\leq Ld_K(x,y)
\qquad\text{for all }x,y\in\Hh.
\]
By the bi-Lipschitz invariance of Hausdorff dimension, \(d_K\) and
\(d_{\CC}\) give the same Hausdorff dimension to every subset of
\(\Hh\).

\noindent We denote by $\tau_p:\Hh\to\Hh$ the left translation by $p \in \Hh$ defined by $\tau_p(q):=p\cdot q$.
We write
\[
\Gamma_{\theta,k}(p):=\tau_p\bigl(\gamma_{\theta,k}([0,1])\bigr)\,.
\]

\subsection{New and existing definitions of Kakeya sets in the first
Heisenberg group}\label{subsec:basic-definitions}
The following proposition identifies the geometric and parametric forms of
the  condition used in Definition \ref{def:cc-kakeya}. It also makes precise the condition in
Lukyanenko's question \cite{LukyanenkoQuestions}:
 \begin{quote}
Is there a Kakeya set in the Heisenberg group? (that is, the set should
have measure zero and contain a left-translate of every unit-length
geodesic starting at the origin)
\end{quote}

\begin{proposition}\label{prop:Lukyanenko-equivalence}
Let \(E\subset\Hh\) be Borel. The following conditions  are equivalent:
\begin{enumerate}
\item[\textup{(i)}]
For every unit-speed minimising CC--geodesic
\(\gamma:[0,1]\to\Hh\) with  \(\gamma(0)=e\), there exists \(p\in\Hh\)
such that
\[
\tau_p\bigl(\gamma([0,1])\bigr)\subset E.
\]
 \item[\textup{(ii)}]
For every \(\theta\in[0,2\pi)\) and every
\(k\in[-2\pi,2\pi]\), there exists \(p\in\Hh\) such that
\[
\Gamma_{\theta,k}(p)\subset E.
\]
\end{enumerate}
 \end{proposition}

 \begin{proof}
This follows immediately from the parametrisation
\(\{\gamma_{\theta,k}:\theta\in[0,2\pi),\ |k|\leq2\pi\}\) above.
\end{proof}

\begin{definition}[CC--geodesic Kakeya set]\label{def:cc-kakeya}
A Borel set \(E\subset\Hh\) is a \emph{CC--geodesic Kakeya set} if, for
every \(\theta\in[0,2\pi)\) and every \(k\in[-2\pi,2\pi]\), there is
\(p\in\Hh\) such that
\begin{equation}\label{eq:def.CCkak}
\Gamma_{\theta,k}(p)\subset E.
\end{equation}
\end{definition}

 \subsubsection{Comparison with existing definitions in the literature}
\label{subsubsec:comparison}

For \(A\subset[-2\pi,2\pi]\), define
\[
\mathcal K(A)
:=
\left\{
E\subset\Hh:
\forall\theta\in[0,2\pi)\ \forall k\in A\ \exists p\in\Hh,
\quad \Gamma_{\theta,k}(p)\subset E
\right\}.
\]

Then
\[
A_1\subset A_2
\quad\Longrightarrow\quad
\mathcal K(A_2)\subset\mathcal K(A_1).
\]
Notice that Definition~\ref{def:cc-kakeya} describes the class
\[
\mathcal K_{\mathrm{CC}}
:=
\left\{E\subset\Hh:E\text{ is Borel and }
E\in\mathcal K([-2\pi,2\pi])\right\}.
\]
F\"assler--Pinamonti--Wald \cite{FPW} formulate the horizontal Kakeya
condition using closed unit segments. After recentering the segment and
changing the translating point, the class they consider is
\(\mathcal K(\{0\})\). Liu \cite{L22} considers the same horizontal
Kakeya problem using open unit segments.\footnote{F\"assler--Pinamonti--Wald
explicitly present \cite[Definition~4.1]{FPW} as the definition recalled
from Liu. The two displayed formulations differ only in the endpoint
convention: \cite[Definition~1.1]{L22} uses open unit segments, whereas
\cite[Definition~4.1]{FPW} uses closed unit segments.}
If \(\mathcal K_{\mathrm L}\) denotes the family of Kakeya sets in
\cite{L22}, then
\begin{equation}\label{eq:K-class-inclusions}
\mathcal K_{\mathrm{CC}}
\subset
\mathcal K(\{0\})
\subset
\mathcal K_{\mathrm L}.
\end{equation}
The second inclusion reflects precisely this endpoint convention.

Liu proves that every \(E\in\mathcal K_{\mathrm L}\) has
\(\dim_{\mathcal H}^{\Hh}(E)\geq3\) \cite{L22}, and
F\"assler--Pinamonti--Wald prove, using different methods,  the same lower bound for every
\(E\in\mathcal K(\{0\})\) \cite{FPW}. The bound is sharp for both
classes, since
\[
\Pi:=\{(x,y,t)\in\Hh:t=0\}
\]
belongs to \(\mathcal K(\{0\})\) and has
\(\dim_{\mathcal H}^{\Hh}(\Pi)=3\). It follows from
\eqref{eq:K-class-inclusions} that every
\(E\in\mathcal K_{\mathrm{CC}}\) satisfies
\begin{equation}\label{eq:lowerbound}
\dim_{\mathcal H}^{\Hh}(E)\geq3.
\end{equation}
However, \(\Pi\notin\mathcal K_{\mathrm{CC}}\), because it contains no
left translate of a nonzero-curvature CC--geodesic segment. Thus the
sharpness of \eqref{eq:lowerbound} for \(\mathcal K_{\mathrm{CC}}\) does
not follow from the horizontal-plane example $\Pi$.

Venieri's Heisenberg Besicovitch problem \cite{Venieri14} uses
ordinary Euclidean segments and is therefore different from the
present CC--geodesic problem. Venieri later developed an abstract
framework for Kakeya-type families in metric spaces
\cite[Definition~3.1]{Venieri17}. Although the condition in
Definition~\ref{def:cc-kakeya} has a similar quantifier structure,
the dimension results in \cite{Venieri17} require additional geometric
axioms for the associated tubes. We do not use that framework here.

\section{Proof of the main results}\label{sec:proof}

\subsection{Proof of Theorem~\ref{thm:main}}
For \((r,\varphi)\in(0,\infty)\times\mathbb R\), define
\[
\Phi(r,\varphi)
:=
\left(
r\cos\varphi,\,
r\sin\varphi,\,
\frac{r^2\varphi}{2}
\right),
\qquad
\Sigma:=\Phi\bigl((0,\infty)\times\mathbb R\bigr).
\]
The  differential \(D\Phi(r,\varphi)\) has rank \(2\) at every point, so
\(\Phi\) is an immersion. Moreover,  \(\Phi\) is injective and its inverse on \(\Sigma\) is
\[
\Phi^{-1}(x,y,t)
=
\left(
\sqrt{x^2+y^2},
\frac{2t}{x^2+y^2}
\right).
\]
Since \(x^2+y^2>0\) on \(\Sigma\), this inverse is the restriction to
\(\Sigma\) of a smooth map defined on the open set
\[
\{(x,y,t)\in\mathbb R^3:x^2+y^2>0\}.
\]
Consequently, \(\Phi\) is a smooth embedding and \(\Sigma\) is a smooth
two-dimensional embedded Euclidean surface in
\(\mathbb R^3\simeq\Hh\).

 \begin{lemma}\label{lem:pgamma}
Define
\(
p:[0,2\pi)\times\bigl([-2\pi,2\pi]\setminus\{0\}\bigr)\to\Hh
\)
by
\begin{equation}\label{eq:p-choice}
p(\theta,k)
:=
\begin{cases}
\left(\dfrac{\cos\theta}{k},\dfrac{\sin\theta}{k},
      \dfrac{\theta}{2k^2}\right),&k>0,\\[8pt]
\left(\dfrac{\cos\theta}{k},\dfrac{\sin\theta}{k},
      \dfrac{\theta+\pi}{2k^2}\right),&k<0.
\end{cases}
\end{equation}
Then, for every \(\theta\in[0,2\pi)\) and every
\(k\in[-2\pi,2\pi]\setminus\{0\}\),
\[
p(\theta,k)\cdot\gamma_{\theta,k}([0,1])\subset\Sigma.
\]
\end{lemma}

\begin{proof}
First let \(k>0\), and write \(p(\theta,k)=(x_0,y_0,t_0)\). From
\eqref{curves} and the group law,
\[
\begin{aligned}
x_0+x_{\theta,k}(s)
&=\frac{\cos(\theta+ks)}{k},\\
y_0+y_{\theta,k}(s)
&=\frac{\sin(\theta+ks)}{k},\\
t_0+t_k(s)
+\frac12\bigl(x_0y_{\theta,k}(s)-y_0x_{\theta,k}(s)\bigr)
&=\frac{\theta+ks}{2k^2}.
\end{aligned}
\]
Hence
\[
p(\theta,k)\cdot\gamma_{\theta,k}(s)
=
\Phi\left(\frac1{|k|},\theta+ks\right)
\in\Sigma.
\]
If \(k<0\), we similarly obtain
\[
p(\theta,k)\cdot\gamma_{\theta,k}(s)
=
\Phi\left(\frac1{|k|},\theta+ks+\pi\right)
\in\Sigma.
\]
This proves the claim.
\end{proof}

\begin{proof}[Proof of Theorem~\ref{thm:main}]
The lower bound follows from \eqref{eq:lowerbound}. To prove sharpness,
set
\[
E:=\Sigma\cup\Pi,
\qquad
\Pi:=\{(x,y,t)\in\Hh:t=0\}.
\]
The set \(E\) is Borel. Fix
\((\theta,k)\in[0,2\pi)\times[-2\pi,2\pi]\). If \(k\neq0\),
Lemma~\ref{lem:pgamma} yields
\[
p(\theta,k)\cdot\gamma_{\theta,k}([0,1])
\subset\Sigma\subset E.
\]
If \(k=0\), then
\[
\gamma_{\theta,0}([0,1])\subset\Pi\subset E.
\]
Thus \(E\) is a CC--geodesic Kakeya set. Both \(\Sigma\) and \(\Pi\) are smooth Euclidean surfaces. Every
\(C^1\) Euclidean surface in \(\Hh\) has Heisenberg Hausdorff dimension
\(3\); see \cite[p.~414]{BDCKMT}. Therefore
\[
\dim_{\mathcal H}^{\Hh}(E)
=
\max\left\{
\dim_{\mathcal H}^{\Hh}(\Sigma),
\dim_{\mathcal H}^{\Hh}(\Pi)
\right\}
=3.
\]
Moreover, \(\mathcal L^3(E)=0\), and the proof is complete. 
\end{proof}

\subsection{Proof of Corollary~\ref{thm:main2}}
For every integer \(m\geq2\), set
\[
r_m:=\frac{m-1}{2m}.
\]
Starting from \(C_m^{(0)}=[0,1]\), construct \(C_m^{(n+1)}\) by
removing from the middle of every component interval of \(C_m^{(n)}\)
an open interval whose length is \(1/m\) times the length of that
component. Each component is therefore replaced by two closed intervals,
each scaled by the factor \(r_m\). Set
\[
C_m:=\bigcap_{n=0}^{\infty}C_m^{(n)}.
\]
Thus \(C_m\) is the middle-\(1/m\) Cantor set. By the standard
dimension formula for self-similar Cantor sets
\cite[Example~4.5]{Fal03}, we have
\[
\dim_{\mathcal H}^{\mathrm{Euc}}(C_m)
=
\frac{\log 2}{\log\!\left(\dfrac{2m}{m-1}\right)}
\longrightarrow1
\qquad\text{as }m\to\infty.
\]
Define
\[
K:=\bigcup_{m=2}^{\infty}C_m,
\qquad
N:=K\times[0,1]^2.
\]

\begin{lemma}\label{lem:N}
The set \(N\) satisfies
\[
\mathcal L^3(N)=0,
\qquad
\dim_{\mathcal H}^{\Hh}(N)=4.
\]
\end{lemma}

\begin{proof}
For each $m \geq 2$, since \(\dim_{\mathcal H}^{\mathrm{Euc}}(C_m)<1\), we have \(\mathcal{L}^1(C_m)=0\). Hence \(\mathcal L^1(K)=0\), and
so \(\mathcal L^3(N)=0\). For every \(m\geq2\), 
\cite[Product Formula~7.2]{Fal03} gives
\[
\begin{aligned}
\dim_{\mathcal H}^{\mathrm{Euc}}(N)
&\geq
\dim_{\mathcal H}^{\mathrm{Euc}}\bigl(C_m\times[0,1]^2\bigr)\\
&\geq
\dim_{\mathcal H}^{\mathrm{Euc}}(C_m)+2.
\end{aligned}
\]
Letting \(m\to\infty\), we obtain
\(\dim_{\mathcal H}^{\mathrm{Euc}}(N)\geq3\). Since
\(N\subset\mathbb R^3\), it follows that
\[
\dim_{\mathcal H}^{\mathrm{Euc}}(N)=3.
\]
The dimension-comparison principle
\cite[Theorem~2.7]{BDCKMT} now gives
\[
\begin{aligned}
\dim_{\mathcal H}^{\Hh}(N)
&\geq
\max\left\{
\dim_{\mathcal H}^{\mathrm{Euc}}(N),
2\dim_{\mathcal H}^{\mathrm{Euc}}(N)-2
\right\}=4.
\end{aligned}
\]
The reverse inequality follows from
\(\dim_{\mathcal H}^{\Hh}(\Hh)=4\), so
\(\dim_{\mathcal H}^{\Hh}(N)=4\), and the proof is complete.
\end{proof}

\begin{proof}[Proof of Corollary~\ref{thm:main2}]
Let \(E=\Sigma\cup\Pi\) be the CC--geodesic Kakeya set constructed in
the proof of Theorem~\ref{thm:main}, and define
\[
F:=E\cup N.
\]
The sets \(E\) and \(N\) are Borel and \(\mathcal L^3\)-null, so the
same is true of \(F\). Since \(E\subset F\) and \(E\) is a
CC--geodesic Kakeya set, \(F\) is also a CC--geodesic Kakeya set.
Finally, Lemma~\ref{lem:N} and \(N\subset F\subset\Hh\) imply
\[
\dim_{\mathcal H}^{\Hh}(F)=4,
\]
and the claim follows.
\end{proof}

\begin{remark}
    The particular choice of the auxiliary set $N$ as in Lemma \ref{lem:N} is not essential, in the sense that if a Borel set  \(N\subset\Hh\) satisfies
\[
\mathcal L^3(N)=0,
\qquad
\dim_{\mathcal H}^{\Hh}(N)=4,
\] 
then the set  \(E\cup N\) is a  CC--geodesic Kakeya set of zero
Lebesgue measure and full Heisenberg Hausdorff dimension. For
completeness, Lemma \ref{lem:N} gives an explicit elementary choice of such a set
\(N\).
\end{remark}

\section{Fixed-curvature Kakeya sets and compact CC--geodesic Kakeya
sets}\label{sec:further-results}

\subsection{Fixed-Curvature Kakeya sets: Definition and elementary properties}

 \begin{definition}[Curvature-\(\kappa\) Kakeya set] \label{def:kappa-Kakeya}
 Fix \(\kappa\in[0,2\pi]\). A Borel set \(E\subset\Hh\) is called a
\emph{curvature-\(\kappa\) Kakeya set} if
\(
E\in\mathcal K(\{\kappa\});
\)
that is, if for every \(\theta\in[0,2\pi)\) there exists \(p\in\Hh\)
such that
\(
\Gamma_{\theta,\kappa}(p)\subset E.
\)
 \end{definition}

 \begin{remark}\label{rem:fixed-curvature}
The following observations clarify Definition~\ref{def:kappa-Kakeya}.
\begin{enumerate}
\item
When \(\kappa=0\), Definition~\ref{def:kappa-Kakeya} recovers the Kakeya sets of F\"assler--Pinamonti--Wald \cite{FPW} and implies Liu's
open-segment condition \cite{L22}.

\item
Fix \(\kappa>0\), and let
\(
\pi_{xy}:\Hh\to\mathbb R^2\)
with 
\(
\pi_{xy}(x,y,t):=(x,y),
\) be the horizontal projection.
Using  \eqref{curves} we see that 
\[
\left(x_{\theta,\kappa}(s)+\frac{\cos\theta}{\kappa}\right)^2
+
\left(y_{\theta,\kappa}(s)+\frac{\sin\theta}{\kappa}\right)^2
=
\frac1{\kappa^2}.
\]
Thus \(\pi_{xy}(\gamma_{\theta,\kappa}([0,1]))\) is a circular arc
with centre
\[
\left(-\frac{\cos\theta}{\kappa},
      -\frac{\sin\theta}{\kappa}\right)
\]
and radius \(1/\kappa\). It is a proper arc when
\(0<\kappa<2\pi\), and it is the whole circle when \(\kappa=2\pi\).
Since a circle of radius \(R\) has Euclidean curvature \(1/R\), this
projected arc has curvature \(\kappa\). For \(\kappa=0\), the projection
of \(\gamma_{\theta,0}([0,1])\) is a straight segment and has curvature
zero.
\item
For  \(0<\kappa\leq2\pi\) we have $\mathcal K(\{\kappa\})
=
\mathcal K(\{-\kappa\})
=
\mathcal K(\{-\kappa,\kappa\}).$  Indeed, a direct calculation from \eqref{curves}
and the group law gives, with angles understood modulo \(2\pi\),
\begin{equation}\label{eq:reverse-geodesic}
\gamma_{\theta,-\kappa}([0,1])
=
\gamma_{\theta,-\kappa}(1)\cdot
\gamma_{\theta-\kappa+\pi,\kappa}([0,1]).
\end{equation}
Hence
\[
\Gamma_{\theta,-\kappa}(p)
=
\Gamma_{\theta-\kappa+\pi,\kappa}
\bigl(p\cdot\gamma_{\theta,-\kappa}(1)\bigr).
\]
Since \(\theta\mapsto\theta-\kappa+\pi\) is a bijection of
\(\mathbb R/(2\pi\mathbb Z)\), and since
\(p\cdot\gamma_{\theta,-\kappa}(1)\) ranges over all of \(\Hh\) as
\(p\) ranges over \(\Hh\), the two  families with curvatures  \(\kappa\) and \(-\kappa\) coincide.
\item
A Borel set \(E\subset\Hh\) is a CC--geodesic Kakeya set if and only if
\[
E\in\mathcal K(\{\kappa\})
\qquad\text{for every }\kappa\in[0,2\pi].
\]
\end{enumerate}
\end{remark}

\subsection{Proof of Theorem~\ref{thm:main3}}

\begin{proof}[Proof of Theorem~\ref{thm:main3}]
Fix \(0<\kappa\leq2\pi\). We extend
\(\widetilde\gamma_{\theta,\kappa}\) to all \(s\in\mathbb R\) using
 formula \eqref{curves}. The extended curve is a smooth
unit-speed horizontal normal geodesic trajectory, although it is no longer
globally minimising after the cut time \(2\pi/\kappa\); see
\cite[\S~6.2, Equation~(24)]{BBN12}. Define
\begin{equation}\label{eq:Kkappa}
K_\kappa
:=
\widetilde\gamma_{0,\kappa}
\left(\left[0,\frac{2\pi}{\kappa}+1\right]\right).
\end{equation}
This set is compact. A direct calculation from \eqref{curves} and the group law of \(\Hh\) gives
\begin{equation}\label{eq:sliding}
\widetilde\gamma_{0,\kappa}(a+s)
=
\widetilde\gamma_{0,\kappa}(a)\cdot
\widetilde\gamma_{\kappa a,\kappa}(s),
\qquad a,s\in\mathbb R,
\end{equation}
where the first parameter \(\kappa a\) is understood modulo \(2\pi\). Given
\(\theta\in[0,2\pi)\), set \(a:=\theta/\kappa\). Then
\(0\leq a<2\pi/\kappa\), and \eqref{eq:sliding} yields
\[
\begin{aligned}
\Gamma_{\theta,\kappa}
\bigl(\widetilde\gamma_{0,\kappa}(a)\bigr)
&=
\widetilde\gamma_{0,\kappa}(a)\cdot
\gamma_{\theta,\kappa}([0,1])\\
&=
\widetilde\gamma_{0,\kappa}(a+[0,1])
\subset K_\kappa,
\end{aligned}
\]
that is  \(K_\kappa\in\mathcal K(\{\kappa\})\). Now set \(I_\kappa=[0,2\pi/\kappa+1]\). The map
\(s\mapsto\widetilde\gamma_{0,\kappa}(s)\) is smooth on the compact
interval \(I_\kappa\), so it 
is Euclidean Lipschitz. Since Lipschitz maps do not increase Hausdorff
dimension, see e.g. \cite[Corollary~2.4(a)]{Fal03}, we have 
\[
\begin{aligned}
\dim_{\mathcal H}^{\mathrm{Euc}}(K_\kappa)
&=
\dim_{\mathcal H}^{\mathrm{Euc}}
\bigl(\widetilde\gamma_{0,\kappa}(I_\kappa)\bigr) \leq \dim_{\mathcal H}^{\mathrm{Euc}}(I_\kappa)
=1.
\end{aligned}
\]
Moreover, $\widetilde\gamma_{0,\kappa}$ is horizontal and has unit horizontal speed.
For \(s<t\) in \(I_\kappa\),
\[
\begin{aligned}
d_{\CC}\bigl(
\widetilde\gamma_{0,\kappa}(s),
\widetilde\gamma_{0,\kappa}(t)
\bigr)
&\leq
L_{\CC}\left(
\left.\widetilde\gamma_{0,\kappa}\right|_{[s,t]}
\right)\\
&=
\int_s^t
\sqrt{\dot x_{0,\kappa}(r)^2+\dot y_{0,\kappa}(r)^2}\,dr\\
&=t-s.
\end{aligned}
\]
Thus $\widetilde\gamma_{0,\kappa}$ is \(1\)-Lipschitz from
\(I_\kappa\) to \((\Hh,d_{\CC})\).  Hence by the Euclidean result, see e.g.  
\cite[Proposition~2.2 and Corollary~2.4(a)]{Fal03}, which applies verbatim to the metric space \((\Hh,d_{\CC})\),  we obtain
\[
\dim_{\mathcal H}^{\Hh}(K_\kappa)\leq1.
\]
On the other hand,
\(
\gamma_{0,\kappa}([0,1])\subset K_\kappa.
\)
Since \(\gamma_{0,\kappa}([0,1])\) is a smooth regular Euclidean
curve and a horizontal \(C^1\) curve,
\[
\dim_{\mathcal H}^{\mathrm{Euc}}
\bigl(\gamma_{0,\kappa}([0,1])\bigr)
=
\dim_{\mathcal H}^{\Hh}
\bigl(\gamma_{0,\kappa}([0,1])\bigr)
=1,
\]
where the second equality follows from
\cite[p.~414]{BDCKMT}. Consequently,
\[
\dim_{\mathcal H}^{\mathrm{Euc}}(K_\kappa)\geq1,
\qquad
\dim_{\mathcal H}^{\Hh}(K_\kappa)\geq1.
\]
Together with the previously proved upper bounds, this completes the
proof.
\end{proof}
The mechanism behind Theorem~\ref{thm:main3} has a simple planar analogue. Recall that a
\(\Gamma\)-Besicovitch set is a planar set \(B\) such that, for every
\(\theta\in[0,2\pi)\), there exists \(v_\theta\in\mathbb R^2\) with
\[
R_\theta\Gamma+v_\theta\subset B,
\]
where \(R_\theta\) denotes counterclockwise rotation through the angle
\(\theta\). Altaf notes in \cite{Altaf24} that, when \(\Gamma\) is a
circular arc, the circle containing \(\Gamma\) is a
\(\Gamma\)-Besicovitch set of Hausdorff dimension \(1\), since
rotations about its centre preserve the circle. Analogously,
\eqref{eq:sliding} shows that \(K_\kappa\) contains a suitable
Heisenberg left translate of
\(\gamma_{\theta,\kappa}([0,1])\) for every
\(\theta\in[0,2\pi)\).

\begin{remark}\label{rem:kneq0} The main difference between the case $\kappa \neq 0$ and $\kappa =0$ lies in the fact that in the former case the horizontal velocity exhausts all unit directions, while in the latter case it is constant, and hence the
horizontal direction remains unchanged. Indeed, 
differentiating the horizontal coordinates from \eqref{curves} we get 
\[
\bigl(\dot x_{\theta,\kappa}(s),\dot y_{\theta,\kappa}(s)\bigr)
   =\bigl(-\sin(\theta+\kappa s),\cos(\theta+\kappa s)\bigr).
\]
If \(\kappa\neq0\), then \(\theta+\kappa s\) runs through all angles as
\(s\) varies over an interval of length \(2\pi/|\kappa|\), so the
horizontal velocity exhausts all unit directions.  If \(\kappa=0\),
the velocity is constant and equal to \((-\sin\theta,\cos\theta)\). 
\end{remark}

\begin{remark}\label{rem:sharp-fixed-curvature}
Fix \(0<\kappa\leq2\pi\). Every
\(E\in\mathcal K(\{\kappa\})\) contains a left translate of
\(\gamma_{0,\kappa}([0,1])\). Left translations are isometries for \(d_{\CC}\) and
bi-Lipschitz maps for the Euclidean metric. Hence they preserve,
respectively, Heisenberg and Euclidean Hausdorff dimension.
Therefore
\[
\dim_{\mathcal H}^{\mathrm{Euc}}(E)\geq1,
\qquad
\dim_{\mathcal H}^{\Hh}(E)\geq1,
\]
and Theorem~\ref{thm:main3} shows that both bounds are sharp. This
conclusion does not extend to \(\kappa=0\): in that case the sharp
 Heisenberg lower bound is \(3\)
\cite{L22,FPW}.
\end{remark}
\subsection{Proof of Theorem~\ref{thm:main4}}
For \(j=0,1,2,\ldots\), set
\[
\rho_j
:=
\begin{cases}
2\pi, & j=0,\\[2mm]
\dfrac{2\pi}{j}, & j\geq1.
\end{cases}
\]
For \(0<\kappa\leq\rho_j\), let
\[
T_\kappa:=\frac{2\pi}{\kappa},
\qquad
\ell_j(\kappa):=
\min\{2,T_\kappa+1-j\},
\]
and set \(\ell_j(0):=2\). \footnote{Notice that \(\ell_j(\kappa)\) is always positive on its stated
domain. Indeed, for \(j\geq1\), the condition
\(0<\kappa\leq\rho_j=2\pi/j\) gives
\(T_\kappa=2\pi/\kappa\geq j\), and hence
\(
1\leq\ell_j(\kappa)\leq2.
\)
.} We see that 
\[
\widetilde\gamma_{0,\kappa}
\bigl([j,j+\ell_j(\kappa)]\bigr) \subset K_\kappa,
\]
with  \(K_\kappa\) as in \eqref{eq:Kkappa}. By  \eqref{eq:sliding} we obtain 
\begin{equation}\label{eq:E_j}
\begin{aligned}
&\widetilde\gamma_{0,\kappa}(j)^{-1}\cdot
\widetilde\gamma_{0,\kappa}
\bigl([j,j+\ell_j(\kappa)]\bigr)\\
&\qquad=
\left\{
\widetilde\gamma_{0,\kappa}(j)^{-1}\cdot
\widetilde\gamma_{0,\kappa}(j+s):
0\leq s\leq\ell_j(\kappa)
\right\}\\
&\qquad=
\left\{
\widetilde\gamma_{j\kappa,\kappa}(s):
0\leq s\leq\ell_j(\kappa)
\right\}
=
\widetilde\gamma_{j\kappa,\kappa}
\bigl([0,\ell_j(\kappa)]\bigr).
\end{aligned}
\end{equation} Here and below, the first parameter \(j\kappa\) is understood modulo
\(2\pi\). We define 
\[
E_j
:=
\bigcup_{0\leq\kappa\leq\rho_j}
\widetilde\gamma_{j\kappa,\kappa}
\bigl([0,\ell_j(\kappa)]\bigr),
\]
where at \(\kappa=0\) we set
\[
\widetilde\gamma_{j\kappa,\kappa}(s)
:=\widetilde\gamma_{0,0}(s)=(0,s,0).
\]
For \(0<\kappa\leq\rho_j\), identity~\eqref{eq:E_j} shows that
\(\widetilde\gamma_{j\kappa,\kappa}([0,\ell_j(\kappa)])\) is the portion of
\(K_\kappa\) parametrised between \(j\) and \(j+\ell_j(\kappa)\), translated so that its initial point is the identity. Finally, let
\[
D_2
:=
\left\{(x,y,0)\in\Hh:x^2+y^2\leq4\right\}
\]
and set
\begin{equation}\label{eq:compact-sharp-set}
E_{\mathrm c}
:=
D_2\cup\bigcup_{j=0}^{\infty}E_j.
\end{equation}
\begin{proof}[Proof of Theorem~\ref{thm:main4}]
  The lower bound follows from Theorem~\ref{thm:main}.  We now 
show that the set \(E_{\mathrm c}\) in
\eqref{eq:compact-sharp-set} is a CC--geodesic Kakeya set.  Fix $\theta\in[0,2\pi)$, $0<k\leq2\pi$, and write
\[
\frac{\theta}{k}=j+u,
\qquad
j:=\left\lfloor\frac{\theta}{k}\right\rfloor,
\qquad
0\leq u<1.
\]
If \(j=0\), then \(k\leq\rho_0=2\pi\); if \(j\geq1\), then
\(jk\leq\theta<2\pi\), and hence \(k<2\pi/j=\rho_j\).  Thus \(k\leq\rho_j\) for this value of \(j\).  Moreover,
\[
u+1\leq2,
\qquad
u+1<T_k+1-j,
\]
and therefore \(u+1\leq\ell_j(k)\).
For \(s\in[0,1]\), applying \eqref{eq:sliding} at the parameter values
\(j\) and \(j+u\) gives
\[
\begin{aligned}
\widetilde\gamma_{0,k}(j)\cdot
\widetilde\gamma_{jk,k}(u+s)
&=\widetilde\gamma_{0,k}(j+u+s)\\
&=\widetilde\gamma_{0,k}(j+u)\cdot
  \widetilde\gamma_{k(j+u),k}(s)\\
&=\widetilde\gamma_{0,k}(j)\cdot
  \widetilde\gamma_{jk,k}(u)\cdot
  \widetilde\gamma_{\theta,k}(s),
\end{aligned}
\]
where we have used that \(k(j+u)=\theta\). Therefore, 
\[
\widetilde\gamma_{jk,k}(u+s)
=
\widetilde\gamma_{jk,k}(u)\cdot
\widetilde\gamma_{\theta,k}(s).
\]
Consequently, with 
\(
p:=\widetilde\gamma_{jk,k}(u),
\)
we have
\[
\Gamma_{\theta,k}(p)
=
p\cdot\gamma_{\theta,k}([0,1])
=
\widetilde\gamma_{jk,k}(u+[0,1])
\subset E_j\subset E_{\mathrm c}.
\]
Thus  condition \eqref{eq:def.CCkak} holds for every \(k\in(0,2\pi]\).  By
Remark~\ref{rem:fixed-curvature}(3), it also holds for every
\(k\in[-2\pi,0)\).  Finally, for \(k=0\),
\[
\Gamma_{\theta,0}(e)
=
\gamma_{\theta,0}([0,1])
\subset D_2\subset E_{\mathrm c},
\]
see also Remark \ref{rem:D_2} later on. We next prove compactness.  The set \(D_2\) is compact.  For each fixed
\(j\), define
\[
A_j
:=
\left\{
(\kappa,s)\in\mathbb R^2:
0\leq\kappa\leq\rho_j,
\quad
0\leq s\leq\ell_j(\kappa)
\right\}.
\]
The function \(\ell_j\) is continuous on \([0,\rho_j]\), since as $\kappa \rightarrow 0$
\[
\ell_j(\kappa)
\rightarrow2=\ell_j(0).
\]
Hence \(A_j\) is compact.  Define
\[
F_j(\kappa,s)
:=
\widetilde\gamma_{j\kappa,\kappa}(s);
\]
that is, for \(\kappa>0\), using \eqref{curves}, its coordinates may be written as
\[
F_j(\kappa,s)
=
\left(
\frac{\cos(\kappa(j+s))-\cos(j\kappa)}{\kappa},
\frac{\sin(\kappa(j+s))-\sin(j\kappa)}{\kappa},
\frac{\kappa s-\sin(\kappa s)}{2\kappa^2}
\right).
\]
Taylor expansions of the three coordinate functions at
\(\kappa=0\) show that the apparent singularities are removable.
Hence \(F_j\) extends smoothly to a neighbourhood of
\([0,\rho_j]\times[0,2]\), with
\[
F_j(0,s)=(0,s,0), 
\]
  and so  the set 
\(
E_j=F_j(A_j)
\)
is compact. It remains to control the sets \(E_j\) as \(j\to\infty\).  Let
\[
q=\widetilde\gamma_{j\kappa,\kappa}(s)=(x,y,t)\in E_j,
\qquad j\geq1.
\]
Since
\(\widetilde\gamma_{j\kappa,\kappa}(0)=e\), the horizontal projection
of \(\widetilde\gamma_{j\kappa,\kappa}|_{[0,s]}\) starts at
\((0,0)\) and ends at \((x,y)\). This projected curve has unit
Euclidean speed, and so the distance between its endpoints is bounded by its
length: 
\[
\begin{aligned}
\sqrt{x^2+y^2}
&\leq
\int_0^s
\sqrt{
\dot x_{j\kappa,\kappa}(r)^2+
\dot y_{j\kappa,\kappa}(r)^2
}\,dr\\
&=s\leq2.
\end{aligned}
\]
Moreover, \(t=0\) when \(\kappa=0\), while for \(\kappa>0\), using
\(r-\sin r\leq r^3/6\) for \(r\geq0\), we obtain
\[
0\leq t
=
\frac{\kappa s-\sin(\kappa s)}{2\kappa^2}
\leq
\frac{\kappa s^3}{12}
\leq
\frac{4\pi}{3j}.
\]
Summarising the above, we obtain  \((x,y,0)\in D_2\) and
\[
\operatorname{dist}_{\mathrm{Euc}}(q,D_2)
\leq
\frac{4\pi}{3j}.
\]
Now let \((q_n)\) be a sequence in \(E_{\mathrm c}\).  If infinitely
many terms lie in \(D_2\), or in one fixed \(E_j\), compactness of that
set gives a convergent subsequence.  Otherwise, after passing to a
subsequence, we may write \(q_n\in E_{j_n}\) with \(j_n\to\infty\).
Choose \(z_n\in D_2\) with
\[
\lvert q_n-z_n\rvert
\leq
\frac{4\pi}{3j_n}.
\]
A subsequence of \((z_n)\) converges to a point \(z\in D_2\), and the
preceding estimate then gives \(q_n\to z\) along the same subsequence.
Thus \(E_{\mathrm c}\) is compact. Since \(F_j\) is \(C^1\) on a neighbourhood of the compact rectangle
\([0,\rho_j]\times[0,2]\),  it is a standard fact that $F_j$ is Lipschitz on  this rectangle. 
 Hence, by
\cite[Proposition~2.2 and Corollary~2.4(a)]{Fal03},
\[
\dim_{\mathcal H}^{\mathrm{Euc}}(E_j)\leq2.
\]
Summarising, by 
\cite[\S~2.2]{Fal03}, we get 
\[
\begin{aligned}
\dim_{\mathcal H}^{\mathrm{Euc}}(E_{\mathrm c})
&=
\max\left\{
\dim_{\mathcal H}^{\mathrm{Euc}}(D_2),
\sup_{j\geq0}\dim_{\mathcal H}^{\mathrm{Euc}}(E_j)
\right\}=2,
\end{aligned}
\]
which in turn, using
\cite[Theorem~2.7]{BDCKMT}, gives
\[
\dim_{\mathcal H}^{\Hh}(E_{\mathrm c})
\leq
\min\left\{
2\dim_{\mathcal H}^{\mathrm{Euc}}(E_{\mathrm c}),
\dim_{\mathcal H}^{\mathrm{Euc}}(E_{\mathrm c})+1
\right\}
=3.
\]
Together with Theorem~\ref{thm:main}, this yields
\[
\dim_{\mathcal H}^{\Hh}(E_{\mathrm c})=3.
\]
The proof is complete.
\end{proof}
\begin{remark}\label{rem:D_2}
   It is easy to see that $D_2$ is a (compact) Kakeya set in $\Hh$ with respect to the definition given in \cite{L22,FPW}.  Indeed, for every
\(\theta\in[0,2\pi)\) and \(s\in[0,1]\),
\[
    \gamma_{\theta,0}(s)
      =(-s\sin\theta,s\cos\theta,0)
\]
and
\(
    (-s\sin\theta)^2+(s\cos\theta)^2=s^2\leq1<4.
\)
Hence
\(
    \Gamma_{\theta,0}(e)
      =\gamma_{\theta,0}([0,1])
      \subset D_2,
\)
so \(D_2\in\mathcal K(\{0\})\).
Moreover, the relative interior
\(
D_2^{\circ}:=\{(x,y,0)\in\Hh:x^2+y^2<4\}
\)
is a smooth two-dimensional Euclidean surface and therefore has Heisenberg
Hausdorff dimension \(3\) by \cite[p.~414]{BDCKMT}. Since
\(D_2^{\circ}\subset D_2\subset\Pi\) and
\(\dim_{\mathcal H}^{\Hh}(\Pi)=3\), it follows that
\[
\dim_{\mathcal H}^{\Hh}(D_2)=3.
\]

\end{remark}

\section*{Funding}
The author is a postdoctoral fellow of the Research Foundation -- Flanders (FWO) under postdoctoral grant no.~1210226N.

\section*{Acknowledgements}
I  would like to thank Marina Iliopoulou, Ioannis Parissis and Tuomas Orponen for valuable discussions and feedback on  this work.  Special thanks are due to Zolt\'an Balogh
for suggesting the definition of curvature-\(\kappa\) Kakeya
sets, which led to the results presented in the first part of 
Section~\ref{sec:further-results}.

\bibliographystyle{plain}

\end{document}